\documentclass[onefignum,onetabnum]{siamart190516}
\nonstopmode

\usepackage{pgfplots}
\pgfplotsset{compat=1.15}
\usepackage{ifpdf} 
\makeatletter
\edef\texforht{TT\noexpand\fi
  \@ifpackageloaded{tex4ht}
    {\noexpand\iftrue}
    {\noexpand\iffalse}}
\makeatother

\usepackage{mathrsfs}
\usepackage{comment}
\usepackage{bm}
\usepackage{amsmath,amssymb}
\usepackage{doi}

\DeclareSymbolFont{EUR}{U}{eur}{m}{n}
\SetSymbolFont{EUR}{bold}{U}{eur}{b}{n}
\DeclareSymbolFontAlphabet{\eur}{EUR}

\newcommand{\scrM}{\mathscr{M}}

\newcommand{\jj}{\mathrm{i}}

\newcommand{\astomega}{\,\ast\sb{\!\!\sb\omega}\,}

\newcommand{\supp}{\mathop{\rm supp}}

\newcommand{\p}{\partial}
\newcommand{\at}[1]{\vert\sb{\sb{#1}}}

\def\R{\mathbb{R}}
\newcommand{\C}{\mathbb{C}}

\newcommand{\upsi}{u}

\newcommand{\N}{\mathbb{N}}

\newcommand{\abs}[1]{\vert #1 \vert}

\newcommand{\norm}[1]{\Vert #1 \Vert}
\newcommand{\sothat}{\,\,{\rm ;}\ \,}

\renewcommand{\aa}[1]{\mathrm{A}\sb{{#1}}}
\newcommand{\bb}[1]{\mathrm{B}\sb{{#1}}}
\renewcommand{\bb}[1]{\mathrm{B}\sb{{#1}}}

\renewcommand{\aa}[1]{\eur{A}\sb{{#1}}}
\renewcommand{\bb}[1]{\eur{B}\sb{{#1}}}

\ifpdf
  \DeclareGraphicsExtensions{.eps,.pdf,.png,.jpg}
\else
  \DeclareGraphicsExtensions{.eps}
\fi

\newsiamremark{remark}{Remark}
\newsiamremark{hypothesis}{Hypothesis}
\crefname{hypothesis}{Hypothesis}{Hypotheses}
\newsiamthm{claim}{Claim}
\newsiamremark{example}{Example}
\newsiamremark{assumption}{Assumption}

\headers{On solutions with compact spectrum}{Andrew Comech}

\title{On solutions with compact spectrum to nonlinear Klein--Gordon
 and Schr\"odinger equations\thanks{Submitted to the editors on 12 April 2021.
}}

\author{Andrew Comech\thanks{Texas A\&M University, College Station, TX, USA;
IITP, Moscow, Russia
  (\email{comech@sdf.org}).}
}

\usepackage{amsopn}

\ifpdf
\hypersetup{
  pdftitle={On solutions with compact spectrum
      to nonlinear Klein--Gordon and Schr\"odinger equations},
  pdfauthor={Andrew Comech}
}
\fi

\begin{document}

\maketitle

% REQUIRED
\begin{abstract}
We consider finite energy solutions
to the nonlinear Schr\"odinger equation
and nonlinear Klein--Gordon equation
and find the condition on the nonlinearity
so that
the standard, one-frequency solitary waves
are the only solutions with compact spectrum.
We also construct an example of a four-frequency
solitary wave solution
to the nonlinear Dirac equation
in three dimensions.
\end{abstract}

\begin{keywords}
Multifrequency solitary waves, compact spectrum, nonlinear Klein--Gordon equation, nonlinear Schr\"odinger equation, soliton resolution conjecture, Titchmarsh convolution theorem
\end{keywords}

\begin{AMS}
35B10, % Periodic solutions 
35B40, % Asymptotic behavior of solutions
35B41, % Attractors 
35Q41, % Time-dependent Schr\"odinger equations, Dirac equations 
37K40 % Soliton theory, asymptotic behavior of solutions
\end{AMS}

\section{Introduction}

This article contains an improvement of
the result obtained in \cite{partial-convolution}
on non-existence
of nontrivial solutions of compact spectrum
to nonlinear Schr\"odinger and Klein--Gordon equations.
We consider nonlinear Schr\"odinger and
nonlinear Klein--Gordon equations,
\begin{eqnarray}\label{eqns}
\jj\p_t\upsi
=
-\Delta\upsi+\alpha(\abs{\upsi}^2)\upsi,
\qquad
-\p_t^2\upsi
=
-\Delta\upsi+m^2\upsi+\alpha(\abs{\upsi}^2)\upsi,
\end{eqnarray}
with $\upsi(x,t)\in\C$, $x\in\R^n$,
$n\in\N$.
The nonlinearity
in \eqref{eqns}
is represented by a function
$\alpha\in C^1(\R,\R)$, $\alpha(0)=0$.
These $\mathbf{U}(1)$-invariant equations
are known to admit solitary wave solutions of the form
\begin{eqnarray}\label{solitary-waves}
\upsi(x,t)=\phi(x)e^{-\jj\omega t},
\qquad
\omega\in\R,
\end{eqnarray}
with $\phi(x)$ decaying at infinity \cite{MR0454365,MR695535}.
Our aim is to prove under which conditions
the one-frequency solitary waves
are the only finite energy solutions with compact spectrum,
defined as follows.

\begin{definition}\label{def-spectrum}
Let $\upsi\in\mathscr{S}'(\R^n\times\R)$,
and let
$\tilde\upsi(x,\omega)=\int\sb{\R} e^{\jj\omega t}\upsi(x,t)\,dt$
be its partial Fourier transform in time.
We say that the spectrum of $\upsi$ is compact if
there is a bounded interval
$I\subset\R$ such that
\[
\supp\tilde\upsi\subset\R^n\times I.
\]
\end{definition}

In \cite{partial-convolution},
we proved that
if the nonlinearity
in \eqref{eqns}
is represented by a function $\alpha(\tau)$
which is either a polynomial or an algebraic function
satisfying certain restrictions,
and moreover satisfying the growth estimate
\[
\abs{\alpha(\tau)}\le C(1+\abs{\tau}^\kappa),
\quad
\forall\tau\ge 0,
\quad
\mbox{with $\kappa$ satisfying}
\quad
\begin{cases}
\kappa>0,&n\le 2,
\\[1ex]
0<\kappa\le 2/(n-2),
&n\ge 3,
\end{cases}
\]
then
the only finite energy solutions with compact spectrum
are solitary waves \eqref{solitary-waves}.
The result was based on the Titchmarsh theorem for partial convolutions
(see \cite[Theorem 2]{partial-convolution}).
In the present article,
we show that under a slightly stronger restriction
\[
0<\kappa<2/(n-2),\qquad n\ge 3,
\]
%(or $n\ge 2$)
the proof can be simplified,
and extend the result to a larger class of algebraic functions $\alpha(\tau)$;
see Theorem~\ref{theorem-one-frequency} below.
In the essence, we prove that
if $0<\kappa<2/(n-2)$, then
solutions $u\in L^\infty(\R,H^1(\R^n))$
with compact spectrum
have improved regularity,
$u\in C^\infty\sb{\mathrm{b}}(\R,L^Q(\R^n)\cap C^{1,a}(\R^n))$,
with any $2\le Q\le\infty$ and any $a\in(0,1)$
(see Theorem~\ref{theorem-better-regularity} below).
This allows one to consider a wider class of algebraic nonlinearities
and to base the argument
on a simple version of the Titchmarsh theorem for partial
convolution in the case of
continuous functions (see Theorem~\ref{theorem-partial} below).

\medskip

We mention that there could be
multifrequency solitary wave solutions
of the form $\sum_{j=1}^{N}\phi_j(x)e^{-\jj\omega_j t}$,
which are known to exist in similar models.
In particular, there are
multifrequency solitary waves in the Klein--Gordon equation
with the mean-field self-interaction \cite{MR2526405}
and with several nonlinear oscillators  \cite{MR2579377}.
Bi-frequency solitary waves can exist in systems of nonlinear
Schr\"odinger equations \cite{PhysRevA.86.053809}
and
in the Soler model and Dirac--Klein--Gordon model
with Yukawa self-interaction
\cite{boussaid2018spectral}.
There are one-, two-, and four-frequency
solitary wave solutions
to the Klein--Gordon equation in discrete time-space
coupled with a nonlinear oscillator \cite{MR3007724}.
In Appendix~\ref{sect-multi-dirac},
we give an example of a four-frequency solitary wave solution
to the nonlinear Dirac equation.

The question of existence of multifrequency
solitary waves and more generally the solutions of compact spectrum
is related to the \emph{soliton resolution conjecture},
which proposes that the long-time asymptotics
of any finite energy solution
to a nonlinear dispersive system
%with $\mathbf{U}(1)$-symmetry
is given by a superposition of outgoing solitary waves
and an outgoing dispersive wave;
see
\cite{MR2032730,MR2275691,MR2304091,MR2308860}
%%(we also mention the probabilistic approach
%% \cite{MR3263670,MR3375596}).
The related results for the nonlinear wave equation
with the critical nonlinearity are obtained in
\cite{duyckaerts2016,duyckaerts2017soliton}.
%% By \cite{MR2032730,MR2308860},
%% this question
%% can be rephrased as the question
%% on the weak global attractor of finite energy solutions.
%% The multifrequency solitary waves,
%% when they exist,
%% also belong to this attractor.
One strategy to attack this problem was proposed in
\cite{MR2032730}:
%% one notices that any solution converges to \emph{radiationless solution},
%% the one that does not lose the energy any more.
%% Then one needs to complete the following two steps:

\medskip

%\begin{verse}
{\it
1.
Prove that any ``omega-limit''
radiationless solution
of finite energy

\ \quad
has a compact spectrum;
%% (see Definition~\ref{def-spectrum});
}

\medskip

%\noindent
{\it
2.
Prove that any solution with compact spectrum
has a spectrum

\ \quad
consisting of a single point,
and hence is a solitary wave.
%% \begin{eqnarray}\label{s-w-1}
%% \upsi(x,t)=
%% \phi(x)e^{-\jj\omega t},
%% \qquad
%% \omega\in\R,
%% \qquad
%% \phi\in H^1(\R^n).
%% \end{eqnarray}
}
%\end{verse}

\medskip

\noindent
Both steps of the program were accomplished
for several models without translation invariance,
such as the Klein--Gordon and Dirac equations
with several nonlinear oscillators
and with the mean field self-interaction
\cite{MR2526405,MR2745798,MR2579377,MR2902120,MR3007724}.
%% and also for the Klein--Gordon equation
%% in the discrete time-space
%% coupled to a nonlinear oscillator
%% \cite{MR3007724}.
See also the review \cite{komech2016attractors}.
%% In other words, in the models mentioned above,
%% \emph{the weak global attractor
%% is formed by solitary waves}.
While presently we can not prove that
any radiationless solution
of a sufficiently general system
has a compact spectrum
(this seems to be a hard task),
in this article we prove that,
under certain assumptions on the nonlinearity,
any solution with a compact spectrum
is a single-frequency solitary wave,
completing the second, easier step of the program
proposed in \cite{MR2032730}.
In particular, our result excludes the existence of
multifrequency solitary waves
under rather general assumptions.

\begin{remark}
There are solutions
to the sine-Gordon equation
known as
\emph{breathers},
exponentially localized in space and are periodic in time,
whose spectrum is not compact;
see e.g.
\cite{PhysRevLett.30.1262}.
We point out that the nonlinearity
in this equation is non-algebraic;
our results
on non-existence of solutions with compact spectrum
do not apply to such systems.
Similarly, the cubic nonlinear Schr\"odinger equation
admits breather-type solutions
\cite{akhmediev1987exact}
with the noncompact spectrum;
their charge and energy are infinite, so again our results
do not cover this case.
\end{remark}

We give the necessary results
on the Titchmarsh theorem for partial convolution
in Section~\ref{sect-titchmarsh}.
%% \ac{there is an interesting paper \cite{MR693042}..}
In Section~\ref{sect-kg},
%we recall the results on the unique continuation principle
%(Theorem~\ref{theorem-ucp})
%and on the global well-posedness for the 
%nonlinear Klein--Gordon equation (Theorem~\ref{theorem-kg-gwp}),
%and then
we derive the regularity results
for the solutions with compact spectrum
(Theorem~\ref{theorem-better-regularity}).
Then in Theorem~\ref{theorem-one-frequency}
we prove that
the nonlinear Schr\"odinger and Klein--Gordon equations
with a certain class of nonlinearities
do not admit multifrequency solitary wave solutions.
An example of a four-frequency solitary wave solution
to the nonlinear Dirac equation
is presented in Appendix~\ref{sect-multi-dirac}.

%\noindent
%{\bf Multifrequency solitary waves.\,}
%If a particular model admits multifrequency solutions,
%defined as exact localized solutions with several frequencies,
%then they also belong to the attractor.

\section{Titchmarsh theorem for partial convolution}
\label{sect-titchmarsh}

The Titchmarsh convolution theorem
\cite{titchmarsh1926zeros}
states that
$\sup\supp\phi\ast\psi
=\sup\supp \phi+\sup\supp\psi$,
for any
$\phi,\,\psi\in\mathscr{E}'(\R)$,
where $\mathscr{E}'(\R)$
is the space of distributions with compact support
(dual to the space $\mathscr{E}(\R)$
which is $C\sp\infty(\R)$ with the seminorms
$\sup\sb{\omega}\abs{f\sp{(k)}(\omega)}$).
We need a version of this theorem for a partial convolution
with respect to only a subset of variables;
for the proofs, we refer to \cite{partial-convolution}.

\begin{lemma}\label{lemma-upper}
For any function $\mu:\R^n\to\R$, $n\ge 1$,
there is a maximal lower semicontinuous function
on $\R^n$
which does not exceed $\mu$;
we denote this function
by $\mu^L(x)$.
Similarly,
there is a minimal upper semicontinuous function
on $\R^n$
which is not exceeded by $\mu$;
we denote this function
by $\mu^U(x)$.
For any $\mu,\,\nu:\R^n\to\R$,
\begin{eqnarray}\label{l-u}
&
\mu^L\le \mu\le\mu^U,
\\
\label{not-in-general}
&
(\mu+\nu)^L\ge\mu^L+\nu^L,
\qquad
(\mu+\nu)^U\le\mu^U+\nu^U.
\end{eqnarray}
\end{lemma}

We recall that the space of distributions
$\mathscr{D}'(\R^n)$ is defined as the dual to
$\mathscr{D}(\R^n)=C^\infty_{\mathrm{comp}}(\R^n)$
(with the standard seminorms),
while
$\mathscr{E}'(\R^n)$ is the space of distributions
with compact support
(the dual to $C^\infty(\R^n)$).

\begin{definition}
Let
$f\in\mathscr{D}'(\R^n\times\R)$.
We define the functions $\aa{f}$ and $\bb{f}$ by
\begin{eqnarray}
\nonumber
&&
\aa{f}:\;\R^n\to\R\sqcup\{\pm\infty\},
\qquad
x\mapsto
%% \inf\big((\{x\}\times\R)\cap\supp f\big);
\inf\big\{
\omega\in\R\sothat
(x,\omega)\in\supp f
\big\};
\\[1ex]
\nonumber
&&
\bb{f}:\;\R^n\to\R\sqcup\{\pm\infty\},
\qquad
x\mapsto
%\sup\big((\{x\}\times\R)\cap\supp f\big).
\sup
\big\{
\omega\in\R\sothat
(x,\omega)\in\supp f
\big\}.
\end{eqnarray}
\end{definition}

\begin{comment}
See Figure~\ref{fig-11}.
\end{comment}

It follows that
$\aa{f}$ is lower semicontinuous,
while
$\bb{f}$ is upper semicontinuous:
\[
\aa{f}(x)=\aa{f}^L(x),
\qquad
\bb{f}(x)=\bb{f}^U(x),
\qquad
\forall x\in\R^n.
\]
\begin{definition}\label{def-sigma}
Let $f\in\mathscr{D}'(\R^n\times\R)$.
We denote by
$\Sigma_f\subset\R^n$ the projection of
$\supp f\subset\R^n\times\R$ onto the first factor:
\[
\Sigma_f
=\big\{x\in\R^n\sothat
(\{x\}\times\R)\cap\supp f\ne\emptyset
\big\}\subset\R^n.
\]
\end{definition}

Thus, the following three statements are equivalent:
\[
x\not\in\Sigma_f
\qquad\Leftrightarrow\qquad
\aa{f}(x)=+\infty
\qquad\Leftrightarrow\qquad
\bb{f}(x)=-\infty.
\]

\begin{lemma}\label{lemma-sigma-closed}
For $f\in \mathscr{D}'(\R^n,\mathscr{E}'(\R))$,
the set $\Sigma_f\subset\R^n$ is closed.
\end{lemma}

%\begin{remark}
%$\Sigma_f$ is not necessarily closed for
%$f\in \mathscr{D}'(\R^n\times\R)$.
%\end{remark}

\begin{lemma}\label{lemma-a-a-b}
For any distribution $f\in\mathscr{D}'(\R^n\times\R)$,
one has:
\begin{eqnarray}\label{a-a-b}
&
\aa{f}(x)\le \aa{f}^U(x)\le \bb{f}(x),
\qquad
\aa{f}(x)\le \bb{f}^L(x)\le \bb{f}(x),
\qquad
\forall x\in\Sigma_f;
\\[2ex]
\label{a-u-l-a}
&
(\aa{f}^U)^L\ge \aa{f},
\qquad
(\bb{f}^L)^U\le \bb{f}.
\end{eqnarray}
\end{lemma}

For $f,\,g\in C^\infty\sb{\mathrm{comp}}(\R^n\times\R)$,
the partial convolution
\[
\astomega:\;
C^\infty\sb{\mathrm{comp}}(\R^n\times\R)
\times
C^\infty\sb{\mathrm{comp}}(\R^n\times\R)
\to
C^\infty\sb{\mathrm{comp}}(\R^n\times\R)
\]
is defined by
\begin{eqnarray}\label{partial-convolution}
(f\astomega g)(x,\omega)
=\int\sb{\R}f(x,\omega-\tau)g(x,\tau)\,d\tau,
\qquad
(x,\omega)\in\R^n\times\R.
\end{eqnarray}
This operation is continuously extended to
$f,\,g\in \mathscr{E}'(\R,C(\R^n))$:
\[
\astomega:\;
\mathscr{E}'(\R,C(\R^n))
\times
\mathscr{E}'(\R,C(\R^n))
\to
\mathscr{E}'(\R,C(\R^n)).
\]

\begin{theorem}
%[Titchmarsh theorem for partial convolution for continuous functions]
\label{theorem-partial}
Let $f,\,g\in \mathscr{E}'(\R,C(\R^n))$.
One has:
\[
\aa{f\astomega g}^U
=
\aa{f}^U+\aa{g}^U
,
\quad\qquad
\bb{f\astomega g}^L
=
\bb{f}^L+\bb{g}^L.
\]
%%\end{enumerate}
\end{theorem}

%% \begin{lemma}\label{lemma-titchmarsh-c}
%% Let $f,\,g\in C(\R^n,\mathscr{E}'(\R))$.
%% Then
%% \[
%% \aa{f\astomega g}^U(x)
%% =
%% \aa{f}^U(x)
%% +
%% \aa{g}^U(x),
%% \qquad
%% \bb{f\astomega g}^L(x)
%% =
%% \bb{f}^L(x)
%% +
%% \bb{g}^L(x).
%% \]
%% \end{lemma}

\begin{proof}
%% Let us prove a similar statement
%% for the case $f,\,g\in \mathscr{E}'(\R,C(\R^n))$.
Since $f$ and $g$ depend continuously on $x$,
the Titchmarsh convolution theorem
can be applied pointwise in $x$,
yielding
\begin{eqnarray}\label{standard-Titchmarsh}
\inf\supp (f\astomega g)(x,\cdot)
=
\inf\supp f(x,\cdot)
+
\inf\supp g(x,\cdot),
\qquad
\forall x\in\R^n,
\end{eqnarray}
and similarly for $\sup$.
Let $f\in \mathscr{E}'(\R,C(\R^n))$
and let $\rho\in\mathscr{D}(\R)$.
If $\mathcal{O}\subset\R^n$ is an open set such that
$\langle \rho,f(x,\cdot)\rangle=0$
for all $x\in \mathcal{O}$,
then, by continuity of $f$ in $x$,
one also has $\langle \rho,f(x,\cdot)\rangle=0$
for all $x$ from the closure of $\mathcal{O}$.
Therefore, given an open set $\Omega\subset\R$,
if $\Omega\cap\supp f(x,\cdot)=\emptyset$ for $x\in \mathcal{O}\subset\R^n$,
then
$\Omega\cap\supp f(x,\cdot)=\emptyset$ for $x$ from the closure of $\mathcal{O}$;
it follows that
\[
\aa{f}^U(x)=\inf\supp f(x,\cdot),
\quad
\bb{f}^L(x)=\sup\supp f(x,\cdot),
\quad
\forall x\ \mbox{such that}\ f(x,\cdot)\not\equiv 0.
\]
Using these relations for each of the terms
in \eqref{standard-Titchmarsh}
(and similarly for $\sup\supp$)
leads to the desired relations.
\end{proof}

\begin{remark}
In a more general case $f,\,g\in \mathscr{E}'(\R,L^2\sb{\mathrm{loc}}(\R^n))$,
there is a result similar to Theorem~\ref{theorem-partial}
(see \cite[Theorem 2]{partial-convolution}):
\[
\aa{f\astomega g}
=
\big(\aa{f}^U+\aa{g}\big)^L
=\big(\aa{f}+\aa{g}^U\big)^L,
\qquad
\bb{f\astomega g}
=
\big(\bb{f}^L+\bb{g}\big)^U
=\big(\bb{f}+\bb{g}^L\big)^U.
\]
\end{remark}

\bigskip

\section{Compact spectrum solutions to nonlinear Klein--Gordon equation}
\label{sect-kg}

\begin{assumption}\label{ass-kappa}
In \eqref{eqns},
the nonlinearity is represented by
$f(u)=\alpha(\abs{u}^2)u$,
with
$\alpha\in C\big(\overline{\R\sb{+}},\R\big)$,
$\alpha(0)=0$,
and there is $C<\infty$
such that
\begin{eqnarray}\label{kappa-such}
%&&
\quad
\abs{\alpha(\tau)}\le C(1+\abs{\tau}^\kappa),
\quad
\forall\tau\ge 0,
\quad
%\\[1ex]
%\nonumber
%&&
%\mbox{with $\kappa$ satisfying}
\mbox{where}
\quad
\begin{cases}
\kappa>0,& n\le 2;
\\[1ex]
0<\kappa<2/(n-2),&
n\ge 3.
\end{cases}
\end{eqnarray}
\end{assumption}

%% \begin{definition}\label{def-algebraic}
%% We will say that $\alpha\in C\big(\overline{\R\sb{+}},\C\big)$
%% is \emph{algebraic}
%% %%of degree $n$
%% if there is
%% a polynomial $\mathcal{P}(\tau,w)$ of $\tau$ and $w$
%% such that
%% for each $\tau\ge 0$
%% one has
%% \[
%% \mathcal{P}(\tau,\alpha(\tau))=0
%% \qquad
%% \forall \tau\ge 0,
%% \]
%% where
%% $\mathcal{P}(\tau,w)$
%% is a finite degree polynomial
%% in $\tau$ and $w$
%% with complex coefficients.
%% \end{definition}

\begin{assumption}\label{ass-algebraic}
$\alpha\in C\big(\overline{\R\sb{+}},\R\big)$ is a non-constant
function
such that there is $J\in\N$
and polynomials
$M_j(\tau)$, $0\le j\le J$,
with $M_J(\tau)\not\equiv 0$
and with
\begin{eqnarray}\label{q-k-j}
\deg M_0>\deg M_j+j,\qquad
\forall j,
\quad 1\le j\le J,
\end{eqnarray}
such that
$w(\tau):=\tau \alpha(\tau)$
satisfies the relation
\[
\mathcal{M}\big(\tau,w(\tau)\big)=0,
\qquad
\forall\tau\ge 0,
\qquad
\mbox{where}
\quad
\mathcal{M}(\tau,w)
:=
\sum_{j=0}^J M_j(\tau)w^j.
\]
We follow the convention that the degree of the zero
polynomial equals $-\infty$,
so that
\eqref{q-k-j}
is formally satisfied
when $M_j(\tau)\equiv 0$ for some $j\ge 1$.
\end{assumption}
%% \begin{remark}
%% \ac{to be checked}
%% If we can prove that
%% $\supp\tilde\upsi
%% =\supp(m^2+\xi^2-\omega^2)\tilde\upsi$,
%% then the sufficient condition would be as follows:
%% \begin{eqnarray}\label{ass-algebraic-all}
%% \mbox{\it There is \ $j\sb\ast\ge 0$, \ $j\sb\ast\le J$, such that}
%% \quad
%% \deg M_{j\sb\ast}\ge\deg M_j
%% \quad
%% \forall j\ne j\sb\ast,
%% \quad
%% 0\le j\le J.
%% \end{eqnarray}
%% \end{remark}

\begin{example}
Let
\[
\alpha(\tau)=\pm A(\tau)^{1/N},
\]
with $N\in\N$, $N\ge 2$,
and with
$A(\tau)=\sum\sb{j=0}^a A_j\tau^j$,
a polynomial with real coefficients of degree $a=\deg A\ge 1$;
if $N$ is even, we additionally assume that
$A(\tau)\ge 0$ for $\tau\ge 0$.
Let
$M_0(\tau)=-(\pm \tau)^N A(\tau)$
and $M_N(\tau)=1$,
with $\deg M_0=a+N$ and $\deg M_N=0$.
Then,
for all $\tau\ge 0$,
\[
\mathcal{M}\big(\tau,\tau \alpha(\tau)\big)
=M_0(\tau)+M_N(\tau)(\tau \alpha(\tau))^N
=
-(\pm \tau)^N A(\tau)\cdot 1+1\cdot(\tau \alpha(\tau))^N=0.
%\qquad\forall\tau\ge 0.
\]
If $n\le 2$,
then both
Assumptions~\ref{ass-kappa}
and~\ref{ass-algebraic}
are satisfied
for any $N\in\N$, $N\ge 2$,
and $a\in\N$.
If $n\ge 3$,
we additionally need
$\kappa=a/N$ to satisfy $\kappa<2/(n-2)$.

We note that in \cite{partial-convolution},
one could only consider the case
$a=1$, $N=2$
(thus $\alpha(\tau)=\pm\sqrt{A_0+A_1\tau}$,
with $A_0\ge 0$ and $A_1>0$), $3\le n\le 6$.
Since in the present article
we do not cover the case $\kappa=\frac{2}{n-2}$,
our results do not apply to this nonlinearity in dimension $n=6$,
the case which is covered by \cite{partial-convolution}.
\end{example}

\begin{example}
Let
\[
\alpha(\tau)
=
\pm\Big(
\frac{A(\tau)}{B(\tau)}\Big)^{1/N},
\]
with
$N\in\N$, and
$A(\tau),\,B(\tau)$ polynomials
with real coefficients
of degrees $a=\deg A\in\N_0$ and $b=\deg B\in\N_0$;
$B(\tau)\ne 0$ for $\tau\ge 0$.
We assume that either
$A(\tau)/B(\tau)\ge 0$ for all $\tau\ge 0$
or that $N$ is odd.
Let
$M_0(\tau)=-(\pm\tau)^N A(\tau)$,
$M_j\equiv 0$ for $1\le j<N$,
$M_N(\tau)=B(\tau)$;
$\deg M_0=a+N$, $\deg M_N=b$.
Then, for all $\tau\ge 0$,
\[
\mathcal{M}\big(\tau,\tau \alpha(\tau)\big)
=
M_0(\tau)
+
M_N(\tau)(\tau \alpha(\tau))^N
=-(\pm \tau)^N A(\tau)\cdot 1+B(\tau)\cdot
(\tau \alpha(\tau))^N=0.
%\qquad\forall\tau\ge 0.
\]
If $n\le 2$, then both Assumption~\ref{ass-kappa}
and Assumption~\ref{ass-algebraic}
are satisfied if $a>b$
(so that \eqref{q-k-j} holds).
If $n\ge 3$, then
we additionally need
$N\in\N$ to be large enough so that
$
\kappa=(a-b)/N$
satisfies
$\kappa<2/(n-2)$.

Let us mention that in \cite{partial-convolution},
besides the case $n\le 2$ (with any $a>b\ge 0$ and any $N\in\N$),
we could only consider the case $a=2$, $b=1$, $N=1$, $n=3$
and the case $a=1$, $b=0$, $N=2$, $3\le n\le 6$.
Since in this article we exclude the case
$\kappa=2/(n-2)$,
our present results
do not apply to the case
$a=1$, $b=0$, $N=2$
(which is again
$\alpha(\tau)=\pm\sqrt{A_0+A_1\tau}$,
with $A_0\ge 0$ and $A_1>0$)
in dimension $n=6$.
\end{example}

\subsection{Regularity}

\begin{theorem}
\label{theorem-better-regularity}
Let
$\upsi\in L^\infty(\R,H^1(\R^n))$, $n\in\N$.
If there is a bounded interval $I\subset\R$
such that $\supp\tilde\upsi\subset\R^n\times I$,
with $\tilde\upsi(x,\omega)$ the Fourier transform of $\upsi$
with respect to time,
then
\begin{eqnarray}\label{u-tt-good}
\upsi\in C\sb{\mathrm{b}}^\infty\big(\R,H^1(\R^n)\big).
\end{eqnarray}
Further, assume that $\upsi$ is a solution
to the nonlinear Schr\"odinger or Klein--Gordon equation
\eqref{eqns}
with some $m>0$ and
with $\alpha(\cdot)$ satisfying Assumption~\ref{ass-kappa}.
Then
\begin{eqnarray}
\label{u-good}
\upsi\in C^\infty\sb{\mathrm{b}}\big(\R,L^Q(\R^n)\cap C^{1,a}(\R^n)\big),
\qquad
\forall Q\in[2,\infty],
\qquad
\forall a\in(0,1).
\end{eqnarray}
\end{theorem}
Above,
\[
C^\infty\sb{\mathrm{b}}(\R)
=\big\{
f\in C^\infty\sb{\mathrm{b}}(\R)
\sothat
\ \sup\sb{t\in\R}\abs{\p_t^j f(t)}<\infty
\quad
\forall j\in\N_0
\big\}
\]
and $C^{1,a}(\R^n)$
is a Banach space with the standard norm
\[
\norm{u}_{C^{1,a}(\R^n)}
=
\sup\sb{x\in\R^n}\abs{u(x)}
+\sup\sb{x\in\R^n}\abs{\nabla u(x)}
+\sup\sb{x,\,y\in\R^n;\,x\ne y}
\frac{\abs{\nabla u(x)-\nabla u(y)}}{\abs{x-y}^a}.
\]

\begin{proof}
We recall that the \emph{quasimeasures}
$\mathscr{Q}(\R)\subset \mathscr{E}'(\R)$
are defined as distributions whose Fourier transform belongs to $L^\infty(\R)$
\cite{MR2032730,MR2308860}.
%% We also recall that
%% operator of multiplication by the function $M\in C(\R)$,
%% $M:\,C^\infty(\R)\to C^\infty(\R)$,
%% $M:\,\mu\mapsto M\mu$,
%% extends to a linear continuous map in the space of quasimeasures
%% if 
Since
$\upsi\in L^\infty(\R,H^1(\R^n))$,
by definition, its Fourier transform is a quasimeasure:
\[
\tilde\upsi(x,\omega)
\in \mathscr{Q}(\R,H^1(\R^n)).
\]
Let $j\in\N$.
Pick $\rho\in C^\infty\sb{\mathrm{comp}}(\R)$
such that
$\rho(\omega)=(-\omega)^j$ for all $\omega\in I$.
The inverse Fourier transform of $\rho$
satisfies
$\check\rho=\mathscr{F}^{-1}[\rho]\in\mathscr{S}(\R)
\subset L^1(\R)$,
hence
\begin{eqnarray}\label{p-t-good}
\p_t^j\upsi(x,t)
=\mathscr{F}^{-1}[\rho(\cdot)\tilde\upsi(x,\cdot)](t)
=\check\rho\ast \upsi(x,\cdot)
\in L^\infty(\R,H^1(\R^n)).
\end{eqnarray}
In other words, $\rho$ is a multiplier in the space of quasimeasures
(for more details, see \cite{MR2308860}).
The inclusion
\eqref{p-t-good}
with any $j\in\N$
shows that
$\upsi\in C^\infty(\R,H^1(\R^n))$,
with $\norm{\p_t^j \upsi(\cdot,t)}_{H^1}$,
for each $j\in\N$, bounded uniformly in $t\in\R$.
This proves \eqref{u-tt-good}.

Let us prove \eqref{u-good}.
There is nothing to do in the case $n=1$
since the inclusion
$
\p_x^2 u=\p_t^2 u+m^2 u+\alpha(\abs{u}^2)u
\in C^\infty\sb{\mathrm{b}}\big(\R,H^1(\R)\big)
$
implies that
%$u\in C^\infty\sb{\mathrm{b}}\big(\R,H^3(\R)\big)$,
%hence
$u
\in C^\infty\sb{\mathrm{b}}\big(\R,H^3(\R)\big)
\subset C^\infty\sb{\mathrm{b}}\big(\R,C^{2,a}(\R)\big)$,
with $a=1/2$.

For $n\ge 2$, the proof of \eqref{u-good} is by induction.
We start with
\eqref{u-tt-good} and use the Sobolev embedding:
\begin{eqnarray}\label{u-good-l-q}
\quad
\upsi\in C^\infty\sb{\mathrm{b}}\big(\R,H^1(\R^n)\big)
\subset C^\infty\sb{\mathrm{b}}\big(\R,L^q(\R^n)\big),
\quad
\begin{cases}
%%2\le q=\infty,&n=1;\\
\forall q\in[2,\infty),&n=2;
\\[1ex]
\forall q\in[2,2n/(n-2)],&n\ge 3.
\end{cases}
\end{eqnarray}
%% This already gives the required inclusion for $n=1$;
%% we now consider $n\ge 2$.
Fix $R>0$.
By \eqref{u-good-l-q},
we may assume that
for each $x_0\in\R^n$
one has
\begin{eqnarray}\label{assume-u-l-q}
\quad
\upsi\in C^\infty\sb{\mathrm{b}}\big(\R,L^q(\mathbb{B}^n_R(x_0))\big),
\quad
\mbox{with some}\quad
q\ge
q_0=\begin{cases}
2+4\kappa,&n=2,
\\[1ex]
2n/(n-2),&n\ge 3,
\end{cases}
\end{eqnarray}
with the norm independent of $x_0\in\R^n$.
Then, in the case when $\upsi$ is a solution to
the nonlinear Klein--Gordon equation
from \eqref{eqns} (for definiteness), one has
\begin{eqnarray}\label{delta-u-good}
\qquad
\Delta\upsi=\p_t^2\upsi+m^2\upsi+\alpha(\abs{\upsi}^2)\upsi
\in C^\infty\sb{\mathrm{b}}\big(\R,L^P(\mathbb{B}^n_R(x_0))\big),
\quad
P:=\frac{q}{1+2\kappa},
\ \ n\ge 2,
%%\begin{cases}
%%1\le P<\infty,&n=2,\\P=\frac{q}{1+2\kappa},&n\ge 3,
%%\end{cases}
\end{eqnarray}
with the norm independent of $x_0$.
Note that our assumptions on $q$
in \eqref{assume-u-l-q}
are such that
\begin{eqnarray}
P=\frac{q}{1+2\kappa}
\ge
\begin{cases}
(2+4\kappa)/(1+2\kappa)=2,&n=2;
\\[1ex]
\displaystyle
\frac{2n/(n-2)}{1+2\kappa}
>
\frac{2n/(n-2)}{1+4/(n-2)}
=\frac{2n}{n+2}>1
,&n\ge 3.
\end{cases}
\end{eqnarray}
Above, for $n\ge 3$, we used
Assumption~\ref{ass-kappa} on $\kappa$.
%% $P\ge 1$ for $n=2$,
%% and similarly for $n\ge 3$ one has
%% \begin{eqnarray}
%% P=\frac{q}{1+2\kappa}
%% >
%% \frac{2n/(n-2)}{1+4/(n-2)}
%% =\frac{2n}{n+2}>1,
%% \qquad
%% n\ge 3.
%% \end{eqnarray}
To arrive at
\eqref{delta-u-good},
we took into account the uniform bound on the time derivative \eqref{u-tt-good}
and the inclusions
$\upsi\in
C^\infty\sb{\mathrm{b}}\big(\R,L^q(\mathbb{B}^n_R(x_0))\big)
\subset
C^\infty\sb{\mathrm{b}}\big(\R,L^P(\mathbb{B}^n_R(x_0))\big)
$
and
$\alpha(\abs{\upsi}^2)\upsi
\in C^\infty\sb{\mathrm{b}}\big(\R,L^{\frac{q}{1+2\kappa}}
(\mathbb{B}^n_R(x_0))\big)$
(by Assumption~\ref{ass-kappa}),
with the norm independent of $x_0$.

%% We note that for $n\le 2$,
%% by \eqref{assume-u-l-q},
%% or rather
%% \eqref{delta-u-good},
%% \[
%% \frac{q}{1+2\kappa}
%% \]

\begin{remark}
If
$n\ge 3$ and $0<\kappa\le 1/(n-2)$,
then
$
P=\frac{q}{1+2\kappa}
\ge\frac{2n/(n-2)}{1+2/(n-2)}=2$
(similarly, if $n=2$, then $q=2+4\kappa$, $P=2$),
and we do not need to restrict the functions
to $\mathbb{B}^n_R(x_0)$:
indeed,
since $2\le P<q$,
one has
$\upsi\in
C^\infty\sb{\mathrm{b}}\big(\R,L^2(\R^n)\cap L^q(\R^n)\big)
\subset
C^\infty\sb{\mathrm{b}}\big(\R,L^P(\R^n)\big)
$,
hence
$
\abs{\alpha(\abs{\upsi}^2)\upsi}
\le
C(\abs{\upsi}+\abs{\upsi}^{1+2\kappa})
\in C^\infty\sb{\mathrm{b}}\big(\R,L^P(\R^n)\big),
$
and then
$
\Delta\upsi
=\p_t^2\upsi+m^2\upsi+\alpha(\abs{\upsi}^2)\upsi
\in C^\infty\sb{\mathrm{b}}\big(\R,L^P(\R^n)\big).
$
\end{remark}

By \eqref{u-good-l-q} and \eqref{delta-u-good},
both $u$ and $\Delta u$ belong to
$C^\infty\sb{\mathrm{b}}(\R,L^P(\mathbb{B}^n_R(x_0)))$,
with
$P=q/(1+2\kappa)$,
with the seminorms
$\sup_{t\in\R}\norm{\p_t^j u}_{L^p(\mathbb{B}^n_R(x_0))}$
and $\sup_{t\in\R}\norm{\p_t^j \Delta u}_{L^p(\mathbb{B}^n_R(x_0))}$,
$j\in\N_0$,
dependent on
$\sup_{t\in\R}\norm{\upsi(t)}_{H^1(\R^n)}$
(and fixed $R>0$), but not on $x_0\in\R^n$,
hence
\[
\upsi\in
C^\infty\sb{\mathrm{b}}\big(\R,H^{2,P}(\mathbb{B}^n_R(x_0))\big)
\cong
C^\infty\sb{\mathrm{b}}\big(\R,W^{2,P}(\mathbb{B}^n_R(x_0))\big),
\quad
P=q/(1+2\kappa)\in[1,+\infty).
\]
The Sobolev embedding gives
\begin{eqnarray}\label{u-l-q}
\upsi\in C^\infty\sb{\mathrm{b}}\big(\R,W^{2,P}(\mathbb{B}^n_R(x_0))\big)
\subset
C^\infty\sb{\mathrm{b}}\big(\R,L^Q(\mathbb{B}^n_R(x_0))\big),
\end{eqnarray}
where we can choose any $Q\in[1,+\infty]$
if $\frac{1}{P}<\frac{2}{n}$.
If, on the contrary,
$\frac{1}{P}\ge\frac{2}{n}$,
then
we can take any $Q\in[1,+\infty)$ satisfying
$
%%\frac{1+2\kappa}{q}
\frac{1}{P}
-\frac{1}{Q}\le \frac{2}{n};
$
we choose $Q$ such that
\begin{eqnarray}\label{q-q}
\frac{1}{Q}=
\frac{1}{P}-\frac{2}{n},
\qquad
\mbox{hence}
\qquad
\frac{1}{Q}=
\frac{1}{q}+\frac{2\kappa}{q}-\frac{2}{n}.
\end{eqnarray}
One can see from \eqref{q-q}
that the inclusion
$\upsi
\in C^\infty\sb{\mathrm{b}}\big(\R,L^Q(\mathbb{B}^n_R(x_0))\big)$
is a strict improvement over 
\eqref{assume-u-l-q}:
\begin{eqnarray}\label{qq}
\frac{1}{q}-\frac{1}{Q}
=\frac{2}{n}-\frac{2\kappa}{q}
\ge
\frac{2}{n}-\frac{2\kappa}{q_0}
=\begin{cases}
1-\frac{2\kappa}{2+4\kappa}\ge \frac 1 2,
&n=2,
\\[1ex]
\frac{2}{n}-\frac{n-2}{n}\kappa>0,
&n\ge 3,
\end{cases}
\end{eqnarray}
with $q_0$ defined in \eqref{assume-u-l-q}.
We note that
the right-hand side of \eqref{qq}
in the case $n\ge 3$
is strictly positive
due to the condition on $\kappa$ from
Assumption~\ref{ass-kappa}.
Now we can return to \eqref{assume-u-l-q}
with $q_{\mathrm{new}}=Q>q$ instead of $q$
and proceed by induction.
Since the right-hand side of \eqref{qq}
does not depend on $q$,
in finitely many steps of the induction argument
we arrive at
$1/P=(1+2\kappa)/q<2/n$,
and then
in \eqref{u-l-q}
we can choose an arbitrary value $Q\in[1,+\infty]$.
Thus,
$\upsi\in C^\infty\sb{\mathrm{b}}\big(\R,L^\infty(\mathbb{B}^n_R(x_0))\big)$,
uniformly in $x_0\in\R^n$,
and hence
$\upsi\in C^\infty\sb{\mathrm{b}}\big(\R,L^\infty(\R^n)\big)$,
which we now interpolate with
\eqref{u-tt-good}.
The inclusion
$\upsi\in C^\infty\sb{\mathrm{b}}\big(\R,L^Q(\R^n)\big)$,
$Q\in[2,+\infty]$,
leads to
\begin{eqnarray}\label{F-good}
\quad
\Delta\upsi=\p_t^2\upsi+m^2\upsi
+\alpha(\abs{\upsi}^2)\upsi
\in C^\infty\sb{\mathrm{b}}\big(\R,L^Q(\R^n)\big),
\qquad
\forall Q\in[2,+\infty],
\end{eqnarray}
hence
$\upsi
\in C^\infty\sb{\mathrm{b}}\big(\R,W^{2,Q}(\R^n)\big)$,
$\forall Q\in[2,+\infty]$,
and by the Sobolev embedding theorem
this leads to
$\upsi\in C^\infty\sb{\mathrm{b}}\big(\R,C^{1,a}(\R^n)\big)$,
with any $a\in(0,1)$.
\end{proof}

\subsection{Reduction to one frequency}

Now we can prove the main result
about the absence of nontrivial solutions
with compact spectrum.
%under rather generic assumptions,
%the only type of solutions with compact spectrum
%is the one-frequency solitary waves.

\begin{theorem}
\label{theorem-one-frequency}
Let $n\in\N$, $m\ge 0$,
and assume that
$\alpha(\cdot)$
satisfies
both
Assumption~\ref{ass-kappa}
and 
Assumption~\ref{ass-algebraic}.
Assume that $\upsi\in L^\infty(\R,H^1(\R^n))$
is a solution to the nonlinear Schr\"odinger or Klein--Gordon equation
\eqref{eqns}.
%% \[
%% -\ddot\upsi=-\Delta\upsi+m^2\upsi-g(\abs{\upsi}^2)\upsi.
%% \]
If there is a bounded interval $I\subset\R$
such that $\supp\tilde\upsi\subset\R^n\times I$,
with $\tilde\upsi(x,\omega)$ the Fourier transform of $\upsi$
with respect to time,
then
\[
\upsi(x,t)=\phi_0(x)e^{-\jj\omega_0 t},
\qquad
\mbox{
with some \ $\phi_0\in H^1(\R^n)$ \ and \ $\omega_0\in\R$}.
\]
\end{theorem}

Note that, in particular, the above theorem applies to finite energy solutions
of the nonlinear Klein--Gordon equation from
%% Theorem~\ref{theorem-kg-gwp}.
\cite[Proposition 2.1]{MR843591}.

\begin{proof}
The proof closely follows
that of \cite[Theorem 6]{partial-convolution},
simplified
in view of Theorem~\ref{theorem-better-regularity};
we provide the proof, shortening the repeating parts.
%The proof for the nonlinear Schr\"odinger equation
%and the nonlinear Klein--Gordon equation is the same;
%for definiteness, we consider the latter case.
Assume that $\upsi\in L^\infty(\R,H^1(\R^n))$
is a solution to
the nonlinear Klein--Gordon equation from
\eqref{eqns}
with compact spectrum,
so that the Fourier transform of $\upsi$ in time,
\[
\tilde\upsi(x,\omega)
=
\int\sb{\R}\upsi(x,t)e^{\jj\omega t}\,dt,
\qquad
\tilde\upsi\in\mathscr{E}'(\R,H^1(\R^n)),
\]
satisfies
$\supp\tilde\upsi\subset\R^n\times [a,b]$,
with some $a,\,b\in\R$, $a<b$.
We denote
\begin{eqnarray}\label{def-sigma-psi}
\Sigma:=\Sigma_{\tilde\upsi}
&=&
\big\{
x\in\R^n\sothat
(\{x\}\times\R)\cap\supp\upsi\ne\emptyset
\big\}
\\
\nonumber
&=&
\big\{
x\in\R^n\sothat
\big(\{x\}\times\R\big)\cap\supp\tilde\upsi\ne\emptyset
\big\}
%%=\cup\sb{j\in\N}\varSigma_j,
\end{eqnarray}
to be the projection of the support of $\upsi$
onto $\R^n$.
Then,
since
%%$\bb{\tilde\upsi}\at{\Sigma}>-\infty$ while
$\supp\tilde\upsi\subset\R^n\times[a,b]$,
\[
\bb{\tilde\upsi}\at{\Sigma}\ge a
\quad\Rightarrow\quad
\bb{\tilde\upsi}^L\at{\Sigma\setminus\p\Sigma}\ge a;
\qquad
\aa{\tilde\upsi}\at{\Sigma}\le b
\quad\Rightarrow\quad
\aa{\tilde\upsi}^U\at{\Sigma\setminus\p\Sigma}\le b.
\]
%% where $\varSigma_j$ are disjoint closed sets.

\begin{lemma}\label{lemma-b-l-a}
$\alpha(\abs{\upsi(x,t)}^2)$ and $\abs{\upsi(x,t)}$
do not depend on time,
and moreover
\[
\bb{\tilde\upsi}^L=\aa{\tilde\upsi},
\qquad
\bb{\tilde\upsi}=\aa{\tilde\upsi}^U,
\qquad
\forall x\in\Sigma.
\]
\end{lemma}

\begin{proof}
Theorem~\ref{theorem-better-regularity}
and \eqref{kappa-such}
show that the function
\[
v(x,t):=\alpha(\abs{\upsi(x,t)}^2)
\]
satisfies
%lead to
\begin{eqnarray}\label{g-is-lq}
v\in C\sb{\mathrm{b}}^\infty\big(\R,C\sb{\mathrm{b}}(\R^n,\R)\big).
%\qquad
%v(x,t):=\alpha(\abs{\upsi(x,t)}^2).
\end{eqnarray}
By \eqref{eqns} and \eqref{g-is-lq},
\begin{eqnarray}\label{w-is-l-q}
(\p_t^2-\Delta+m^2)\upsi
=-\alpha(\abs{\upsi}^2)\upsi\in C\sb{\mathrm{b}}^\infty\big(\R,C\sb{\mathrm{b}}(\R^n)
\big).
\end{eqnarray}
%% with any $Q\in[2,+\infty]$.
%% \ac{Simply $Q=\infty???$}
We apply the Fourier transform to \eqref{w-is-l-q};
denoting by
$\tilde v(x,\omega)$
the Fourier transform of $v(x,t):=\alpha(\abs{\upsi(x,t)}^2)$
in time, one has
\begin{eqnarray}\label{bar-psi-psi}
(m^2-\omega^2-\Delta)\tilde\upsi
=-\tilde v\astomega\tilde\upsi.
\end{eqnarray}
Multiplying \eqref{w-is-l-q}
by $\bar\upsi$,
we have:
%\begin{eqnarray}\label{bar-psi-psi}
\begin{eqnarray}\label{psi-psi}
\bar\upsi(m^2+\p_t^2-\Delta)\upsi
=-\abs{\upsi}^2 \alpha(\abs{\upsi}^2)
\in C\sb{\mathrm{b}}^\infty\big(\R,C\sb{\mathrm{b}}(\R^n)\big).
\end{eqnarray}
Let $\mathcal{M}$ be as in Assumption~\ref{ass-algebraic}.
Applying $\mathcal{M}(\abs{\upsi}^2,\cdot)$
to both sides of the relation \eqref{psi-psi}
leads to
\begin{eqnarray}\label{leads-to}
0=\mathcal{M}
\big(
\abs{\upsi}^2,
\abs{\upsi}^2 \alpha(\abs{\upsi}^2)
\big)
&=&
\mathcal{M}
\big(
\abs{\upsi}^2,
-\bar\upsi(m^2-\omega^2-\Delta)\upsi
\big)
\\
\nonumber
&=&
\sum\sb{j=0}\sp J
M_j(\abs{\upsi}^2)
\big(
-\bar\upsi(m^2+\p_t^2-\Delta)\upsi
\big)^j.
\end{eqnarray}
We note that
$\tilde\upsi\sp\sharp\astomega\tilde\upsi=\widetilde{\abs{\upsi}^2}$,
where
\begin{eqnarray}\label{def-sharp}
f\sp\sharp(x,\omega)=\overline{f(x,-\omega)}.
\end{eqnarray}

%% \begin{lemma}\label{lemma-missing}
%% \ac{missing piece, better prove this nicely:}
%% \[
%% \supp
%% (\tilde\upsi\sp\sharp\astomega(m^2-\omega^2-\Delta)\tilde\upsi)
%% \subset
%% \supp(\tilde\upsi\sp\sharp\astomega\tilde\upsi).
%% \]
%% \end{lemma}

\begin{lemma}\label{lemma-small-supp}
$
\ \bb{\tilde\upsi\sp\sharp\astomega(m^2-\omega^2-\Delta)\tilde\upsi}(x)
\le
\bb{\tilde\upsi\sp\sharp\astomega\tilde\upsi}(x)
\ $
$\forall x\in\R^n$.
\end{lemma}

\begin{proof}
Since
$
\supp
(m^2-\omega^2-\Delta)\tilde\upsi
\subset
\supp
\tilde\upsi,
$
there is the inequality
\begin{eqnarray}\label{bb}
\bb{(m^2-\omega^2-\Delta)\tilde\upsi}(x)
\le
\bb{\tilde\upsi}(x),
\qquad
\forall x\in\R^n.
\end{eqnarray}
Therefore,
applying twice the Titchmarsh theorem for partial convolution
(Theorem~\ref{theorem-partial})
and \eqref{bb},
we derive:
\begin{eqnarray*}
\bb{\tilde\upsi\sp\sharp\astomega\tilde\upsi}
&=&
\big(\bb{\tilde\upsi\sp\sharp}^L+\bb{\tilde\upsi}\big)^U
%% \ge
%% \big(
%% \bb{\tilde\upsi\sp\sharp}^L+\bb{(m^2-\omega^2-\Delta)\tilde\upsi}\big)^U
\ge
\big(
\bb{\tilde\upsi\sp\sharp}^L+\bb{(m^2-\omega^2-\Delta)\tilde\upsi}^L\big)^U
\\[1ex]
&=&
(\bb{\tilde\upsi\sp\sharp\astomega(m^2-\omega^2-\Delta)\tilde\upsi}^L)^U
=
\bb{\tilde\upsi\sp\sharp\astomega(m^2-\omega^2-\Delta)\tilde\upsi}.
%\qedhere
\end{eqnarray*}
For the last equality, we used Lemma~\ref{lemma-a-a-b}
\end{proof}

Now we apply Theorem~\ref{theorem-partial} to the Fourier transform
(in time)
of the relation \eqref{leads-to}
and use Assumption~\ref{ass-algebraic},
getting
$
\bb{\tilde\upsi\sp\sharp\astomega(m^2-\omega^2-\Delta)\tilde\upsi}^L
\le 0;
$
then
$\bb{\tilde\upsi\sp\sharp\astomega(m^2-\omega^2-\Delta)\tilde\upsi}\le 0$,
and similarly
$\aa{\tilde\upsi\sp\sharp\astomega(m^2-\omega^2-\Delta)\tilde\upsi}\ge 0$.
It follows that
\begin{eqnarray}\label{i-f-t}
\supp\tilde\upsi\sp\sharp\astomega(m^2-\omega^2-\Delta)\tilde\upsi\subset\R^n\times\{0\}.
\end{eqnarray}

\begin{lemma}
\label{lemma-now-lemma}
$\abs{\upsi}^2 \alpha(\abs{\upsi}^2)$ is time-independent.
\end{lemma}

\begin{proof}
By \eqref{i-f-t},
the function
$
G(x,t):=\abs{\upsi}^2 \alpha(\abs{\upsi}^2)
=\bar u(m^2+\p_t^2-\Delta)u
$
satisfies
$
\supp \tilde G(x,\omega)\subset\R^n\times\{0\}.
$
We conclude that
\begin{eqnarray}\label{g-powers}
\tilde G(x,\omega)=\sum_{j\in\N\sb{0}}\delta^{(j)}(\omega) G_j(x),
\qquad
x\in\R^n,
\quad \omega\in\R.
\end{eqnarray}
Above,
in agreement with the general theory of distributions
\cite{MR717035},
the summation in $j\in\N\sb{0}$ is locally finite:
for each compact subset $K\subset\R^n$,
%and $x\in K$,
there are finitely many terms $J(K)\in\N$ in the
restriction of \eqref{g-powers} onto $K$
(cf. \cite[Theorem 2.3.5]{MR717035}).
Moreover, the terms with derivatives of $\delta(\omega)$
should not appear in \eqref{g-powers}.
Indeed, the Fourier transform
of \eqref{g-powers} coupled with a test function
$\varphi\in\mathscr{D}(K)$
is
\begin{eqnarray}\label{g-powers-t}
\langle\varphi,G(\cdot,t)\rangle=
\frac{1}{2\pi}
\sum_{j=0}^{J(K)} (\jj t)^j
\langle\varphi, G_j\rangle,
\qquad
x\in K,
\quad t\in\R;
\end{eqnarray}
if $G_j\ne 0$ for some $1\le j\le J(K)$,
then for some nonzero test function
the right-hand side of \eqref{g-powers-t}
would be growing in time,
in contradiction to \eqref{g-is-lq}.
This implies that in \eqref{g-powers} the only nonzero term is the one
with $j=0$.
Thus, $G(x,t)=\abs{\upsi(x,t)}^2\alpha(\abs{\upsi(x,t)}^2)=G_0(x)$ does not depend on time.
%%hence, by the argument after \eqref{g-powers},
\end{proof}

Let us argue that
since
by Lemma~\ref{lemma-now-lemma}
the expression $\abs{\upsi}^2 \alpha(\abs{\upsi}^2)$ is time-inde\-pendent,
so is $\abs{\upsi}^2$.
%We note that
%$\tau \alpha(\tau)$ is a nonconstant function of $\tau$.
%We can not have $\tau\alpha(\tau)=0$ for all $\tau\ge 0$
%since $\alpha(\tau)$ is a nonconstant function.
Since $\tau \alpha(\tau)\equiv C\in\R$,
we substitute this value into $\mathcal{M}$,
arriving at
\begin{eqnarray}\label{zero-is-m}
0
\equiv\mathcal{M}\big(\tau,\tau\alpha(\tau)\big)
=\mathcal{M}\big(\tau,C\big)
=\sum\sb{j=0}^J M_j(\tau)C^j.
\end{eqnarray}
Due to the conditions
\eqref{q-k-j},
the right-hand side of \eqref{zero-is-m}
is a polynomial of degree $\deg M_0>0$;
thus,
for each $(x,t)\in\R^n\times\R$,
the value $\tau=\abs{\upsi(x,t)}^2$
has to be equal to one of the roots of
$\scrM(\tau,C)$;
due to the continuous dependence
of $\abs{\upsi(x,t)}^2$ of $x$ and $t$
(Theorem~\ref{theorem-better-regularity}),
it has to be the same root
for all $(x,t)\in\R^n\times\R$.
Thus,
$\abs{\upsi(x,t)}^2$ also does not depend on time,
resulting in
\begin{eqnarray}\label{zero-is-b}
\supp\widetilde{\abs{\upsi}^2}\subset\R^n\times\{0\}.
\end{eqnarray}
Using the above relation and
applying Theorem~\ref{theorem-partial}
to
$\widetilde{\abs{\upsi}^2}=
\tilde\upsi\sp\sharp\astomega\tilde\upsi$
(with $\tilde\upsi\sp\sharp$
defined according to \eqref{def-sharp}),
we conclude that
\[
0
=\bb{\widetilde{\abs{\upsi}^2}}(x)
\ge
\bb{\tilde\upsi}^L(x)+\bb{\tilde\upsi\sp\sharp}(x)
=\bb{\tilde\upsi}^L(x)-\aa{\tilde\upsi}(x),
\qquad
\forall x\in\Sigma.
\]
Thus,
$\bb{\tilde\upsi}^L\le\aa{\tilde\upsi}$
for all $x\in\Sigma$.
On the other hand, by Lemma~\ref{lemma-a-a-b},
$\bb{\tilde\upsi}^L\ge \aa{\tilde\upsi}$
for all $x\in\Sigma$.
We conclude that
\[
\bb{\tilde\upsi}^L=\aa{\tilde\upsi}
\quad
\mbox{and similarly}
\quad
\bb{\tilde\upsi}=\aa{\tilde\upsi}^U,
\qquad
\forall x\in\Sigma.
%\qedhere
\]
\end{proof}

The rest of the proof repeats
\cite{partial-convolution},
to which we refer for more details
and only give a short sketch.
By Lemma~\ref{lemma-b-l-a},
\[
V(x):=v(x,t)
=
\alpha(\abs{\upsi(x,t)}^2)
\quad
\mbox{does not depend on time;}
\qquad
\tilde v(x,\omega)=2\pi\delta(\omega)V(x).
\]
Then equation \eqref{w-is-l-q}
takes the form
\begin{eqnarray}\label{psi-tilde-ucp}
\Delta\tilde\upsi
=
m^2\tilde\upsi
-\omega^2\tilde\upsi
+V(x)\tilde\upsi,
\end{eqnarray}
where
$V(x)=\alpha(\abs{\upsi(x,t)}^2)$,
$V\in C\sb{\mathrm{b}}(\R^n,\R)$
by Theorem~\ref{theorem-better-regularity}.
Since $V(x)$ is sufficiently regular,
%satisfies conditions of Theorem~\ref{theorem-ucp},
one applies the unique continuation property
for the Laplace operator
(see e.g. \cite{MR1809741})
to an $L^2$-function
$\tilde\upsi$
(valued in $\mathscr{D}'(\Omega)$)
which solves
\eqref{psi-tilde-ucp},
concluding that the set $\Sigma_{\tilde\upsi}$
from \eqref{def-sigma-psi}
has to be the whole space
or else $\tilde\upsi$ is identically zero.
Again using the unique continuation property,
one proves the inclusions
\[
\supp\tilde\upsi\subset\R\times(-\infty,\inf\bb{\tilde\upsi}],
\qquad
\supp\tilde\upsi\subset\R\times[\sup\aa{\tilde\upsi},+\infty).
\]
By Lemma~\ref{lemma-b-l-a},
one has $\bb{\tilde\upsi}^L=\aa{\tilde\upsi}$,
and then it follows that
$\supp \tilde u\subset\R^n\times\{\omega_0\}$,
with some $\omega_0\in I$.
Similarly to the proof of Lemma~\ref{lemma-now-lemma},
this leads to $\upsi(x,t)=e^{-\jj\omega_0 t}\phi(x)$,
concluding the proof of Theorem~\ref{theorem-one-frequency}.
\end{proof}

%% \begin{remark}
%% Alternatively,
%% from the energy conservation
%% and the assumption $\psi\in H^1(\R^n)$,
%% we seem to have
%% \[
%% \psi\in L^{2k+2}(\R^n),
%% \qquad
%% k=\deg f.
%% \]
%% Then
%% $
%% g(\abs{\psi}^2)\in L^{\frac{2k+2}{2k}}(\R^n).
%% $
%% For the inclusion $v:=g(\abs{\psi}^2)\in L^{n/2}(\R^n)$
%% we seem to need
%% \[
%% \frac{2k+2}{2k}\ge p=\frac{n}{2},
%% \qquad
%% 2+\frac{2}{k}\ge n,
%% \qquad
%% \frac{2}{n-2}\ge k.
%% \]
%% \end{remark}

\appendix

\section{Multifrequency solitary waves of nonlinear Dirac equation}
\label{sect-multi-dirac}

Here we give an explicit construction
of multifrequency solitary waves for
the nonlinear Dirac equation with the scalar self-interaction,
known as the Soler model
\cite{jetp.8.260,PhysRevD.1.2766}:
\begin{eqnarray}\label{nld}
\jj\p_t\psi
=D_m\psi-f(\psi\sp\ast\beta\psi)\beta\psi,
\qquad
\psi(x,t)\in\C^4,
\quad
x\in\R^3,
\end{eqnarray}
where the free Dirac operator
$D_m:\,L^2(\R^3,\C^4)\to L^2(\R^3,\C^4)$
with domain $\mathfrak{D}(D_m)=H^1(\R^3,\C^4)$
is given by
\[
D_m=-\jj\bm\alpha\cdot\nabla+\beta m,
%=-\jj\begin{bmatrix}0&\sigma\cdot\nabla
%\\\sigma\cdot\nabla&0\end{bmatrix}+m\beta,
\]
with
%$\bm\alpha=(\alpha^1,\alpha^2,\alpha^3)$
%with
$\alpha^i=\begin{bmatrix}0&\sigma_i\\\sigma_i&0\end{bmatrix}$,
$\sigma_i$, $1\le i\le 3$ the Pauli matrices,
and
$\beta=\begin{bmatrix}I_{\C^2}&0\\0&-I_{\C^2}\end{bmatrix}$.

\begin{assumption}\label{ass-mono}
\begin{enumerate}
\item
\label{ass-mono-1}
There is
$V\in C^\infty(\R^3)$
which is positive, spherically symmetric, and strictly monotonically decreasing,
with $\lim\sb{\abs{x}\to\infty}V(x)=0$,
such that there are eigenfunctions of the Dirac operator
$D_m+\beta V$
corresponding to eigenvalues
$\omega_0$ and $\omega_1$,
with $0<\omega_0<\omega_1<m$:
\begin{eqnarray}\label{omega-phi}
\omega_j\phi_j=D_m\phi_j-\beta V\phi_j,
\qquad
j=0,\,1.
\end{eqnarray}
\item
\label{ass-mono-2}
The corresponding eigenfunctions have the form
\begin{eqnarray}\label{phi-j}
\quad
\phi_j(x)=\begin{bmatrix}v_j(r)\bm{n}_j\\\jj u_j(r)\sigma_r\bm{n}_j\end{bmatrix},
\qquad
\bm{n}_j\in\C^2,
\quad
\abs{\bm{n}_j}=1,
\quad
j=0,\,1,
\end{eqnarray}
with $r=\abs{x}$, $\sigma_r=r^{-1}x\cdot\sigma$ (for $x\ne 0$),
and with $v_j$, $u_j$ smooth (considered as functions of $x\in\R^3$)
and real-valued;
\item
\label{ass-mono-3}
The function $v_0$ is strictly positive;
\item
\label{ass-mono-4}
$\rho(r):=|v_0(r)|^2-|u_0(r)|^2$ is monotonically decreasing
in $r\ge 0$.
\end{enumerate}
\end{assumption}

Thus, $v_j$ and $u_j$ satisfy the system
\begin{eqnarray}\label{v-u-equations}
\begin{cases}
\omega_j v_j=\p_r u_j+\frac{n-1}{r}u_j+(m-V)v_j,
\\[1ex]
\omega_j u_j=-\p_r v_j-(m-V)u_j,
\end{cases}
\qquad
j=0,\,1.
\end{eqnarray}

\begin{remark}\label{remark-bif}
Let us sketch a construction of states
$\phi_j$
satisfying Assumption~\ref{ass-mono}.
We assume that there is 
a spherically symmetric potential $W(x)>0$
which we consider as a function of $r=\abs{x}$,
monotonically decreasing at infinity with
$\lim\sb{r\to\infty}W(r)\to 0$,
such that
$H=-\frac{1}{2m}\Delta-W$ has
a groundstate eigenvalue $E_0=-1$
and an excited state with eigenvalue $E_1=-1/2$,
with the corresponding eigenfunctions
$\varphi_0$ and $\varphi_1$:
\begin{eqnarray}\label{one-and-a-half}
\qquad
-\varphi_0=-\frac{1}{2m}\Delta\varphi_0-W\varphi_0,
\quad
-\frac 1 2\varphi_1=-\frac{1}{2m}\Delta\varphi_1-W\varphi_1,
\quad
\varphi_0,\,\varphi_1\in L^2(\R^n).
\end{eqnarray}
Since $\varphi_0$ is a groundstate,
we can assume that it is spherically symmetric,
strictly positive,
and monotonically decreasing to zero at infinity.
For $\omega\in(0,m)$, for $r\ge 0$, we define
\[
V(r)=(m-\omega)W\big((m-\omega)^{1/2}r\big),
\]
\[
\hat v_j(r)=\varphi_j\big((m-\omega)^{1/2}r\big),
\qquad
\hat u_j(r)=-\frac{\p_r\hat v_j(r)}{2m}
=-\frac{(m-\omega)^{1/2}}{2m}\varphi_j'\big((m-\omega)^{1/2}r\big),
\]
where $j=0,\,1$.
Then $(\hat v_0,\hat u_0)$ and $(\hat v_1,\hat u_1)$
satisfy the relations
\[
-(m-\omega)\hat v_0
=-\frac{\Delta \hat v_0}{2m}-V \hat v_0
=\p_r \hat u_0+\frac{n-1}{r}\hat u_0-V \hat v_0,
\]
\[
-\frac{m-\omega}{2}\hat v_1
=-\frac{\Delta \hat v_1}{2m}-V \hat v_1
=\p_r \hat u_1+\frac{n-1}{r}\hat u_1-V\hat v_1,
\]
hence
\[
\omega \hat v_0
=\p_r \hat u_0+\frac{n-1}{r}\hat u_0
+(m-V)\hat v_0,
\qquad
\frac{m+\omega} 2 \hat v_1
=\p_r \hat u_1+\frac{n-1}{r}\hat u_1
+(m-V)\hat v_1,
\]
which coincide with the first equation in \eqref{v-u-equations}
(with $\omega$ and $(m+\omega)/2$ in place of $\omega_j$).
The function $\varphi_0(r)$ is strictly monotonically decreasing,
while
$\mathop{\lim\sup}\limits\sb{r\to\infty}
\abs{\varphi_0'(r)/\varphi_0(r)}<\infty$
(one can show that $f(r)=\varphi_0'(r)/\varphi_0(r)$
satisfies
$f'(r)=1-W(r)-(n-1)f/r-f^2$, $r>1$,
$f(1)=\varphi_0'(1)/\varphi_0(1)<0$;
solutions to this equation
are either uniformly bounded from below or
approach $-\infty$
as $r\to r_0-0$, with some $r_0\in (1,\infty)$).
Therefore, the function
\begin{eqnarray}\label{rho-mono}
\qquad
\hat\rho(r)
:=|\hat v_0(r)|^2-|\hat u_0(r)|^2
%\\\nonumber
=
\big|\varphi_0((m-\omega)^{1/2}r)\big|^2
-\frac{m-\omega}{4m^2}\big|\varphi_0'((m-\omega)^{1/2}r)\big|^2
\end{eqnarray}
is also monotonically decreasing
as long as $\omega$ is sufficiently close to $m$.
The perturbation theory
allows one to start with
\[
\omega
\begin{bmatrix}\hat v_0\\\hat u_0\end{bmatrix}
=
\begin{bmatrix}m-V&\p_r+\frac{n-1}{r}
\\-\p_r&-2m+\omega\end{bmatrix}
\!
\begin{bmatrix}\hat v_0\\\hat u_0\end{bmatrix},
\quad
\frac{m+\omega}{2}
\begin{bmatrix}\hat v_1\\\hat u_1\end{bmatrix}
=
\begin{bmatrix}m-V&\p_r+\frac{n-1}{r}
\\-\p_r&-2m+\frac{m+\omega}{2}\end{bmatrix}
\!
\begin{bmatrix}\hat v_1\\\hat u_1\end{bmatrix}
\]
and to construct eigenfunctions
$(v_0,u_0)\in L^2(\R^n,\C^2)$
and
$(v_1,u_1)\in L^2(\R^n,\C^2)$
to
\eqref{v-u-equations},
\[
\omega_0
\begin{bmatrix}v_0\\u_0\end{bmatrix}
=
\begin{bmatrix}m-V&\p_r+\frac{n-1}{r}
\\-\p_r&-m+V\end{bmatrix}
\begin{bmatrix}v_0\\u_0\end{bmatrix},
\qquad
\omega_1
\begin{bmatrix}v_1\\u_1\end{bmatrix}
=
\begin{bmatrix}m-V&\p_r+\frac{n-1}{r}
\\-\p_r&-m+V\end{bmatrix}
\begin{bmatrix}v_1\\u_1\end{bmatrix}
\]
as bifurcations from 
$(\hat v_0,\hat u_0)$
and
$(\hat v_1,\hat u_1)$,
with corresponding eigenvalues
$\omega_0\approx\omega$
and $\omega_1\approx (m+\omega)/2$,
as long as $\omega\lesssim m$
is chosen sufficiently close to $m$.
(We note that the operators
in the right-hand sides are self-adjoint
in
$H^1_{\mathrm{even,\,odd}}(\R,|r|^{n-1}\,dr;\C^2)$,
the space
consisting of $\C^2$-valued $H^1$-functions
with respect to the measure $\abs{r}^{n-1}\,dr$
on $\R$,
with their
first component being even and the second
being odd as functions of $r\in\R$;
see e.g. \cite{boussaid2017nonrelativistic}.)
The analysis shows that
$(v_j,u_j)$ can also be chosen real.
Moreover, if $\omega$ is sufficiently close to $m$,
then
$\rho(r):=\abs{v_0(r)}^2-\abs{u_0(r)}^2$
(cf. \eqref{rho-mono})
is monotonically decreasing.
\end{remark}

Similarly to \eqref{omega-phi}, we have
\begin{eqnarray}\label{omega-phi-1}
-\omega_j\chi_j=D_m\chi_j-\beta V\chi_j,
\qquad
j=0,\,1,
\end{eqnarray}
with
\begin{eqnarray}\label{chi-j}
\chi_j(x)=\begin{bmatrix}-\jj u_j(r)\sigma_r\bm{m}_j\\v_j(r)\bm{m}_j\end{bmatrix},
\qquad
\bm{m}_j\in\C^2,
\qquad
\abs{\bm{m}_j}=1,
\qquad
j=0,\,1.
\end{eqnarray}
(These expressions for $\chi_j$
can be obtained from $\phi_j$
by applying to \eqref{phi-j}
the charge conjugation
operator $\jj\gamma^2\bm{K}$, with $\gamma^2=\beta\alpha^2
=\begin{bmatrix}0&\sigma_2\\-\sigma_2&0\end{bmatrix}$,
with
$\bm{K}:\,\C^N\to\C^N$ the complex conjugation;
for more details, see \cite{boussaid2018spectral} or \cite{opus}).
It follows that
for any $a_j,\,b_j\in\C$, $j=0,\,1$,
the function
\begin{eqnarray}\label{def-psi-a-b}
\quad
\psi(x,t)=
a_0\phi_0(x)e^{-\jj\omega_0 t}
+
a_1\phi_1(x)e^{-\jj\omega_1 t}
+
b_0\chi_0(x)e^{\jj\omega_0 t}
+
b_1\chi_1(x)e^{\jj\omega_1 t}
\end{eqnarray}
satisfies the linear Dirac equation
$
\jj\p_t\psi
=D_m\psi-\beta V\psi.
$

We note that
$\phi_j$ and $\chi_j$, $j=0,\,1$,
defined in \eqref{phi-j} and \eqref{chi-j},
satisfy
\[
\chi_i^*\beta\phi_j
=\phi_i^*\beta\chi_j=0,
\qquad
i,\,j=0,\,1,
\]
for any choice of $\bm{n}_j$, $\bm{m}_j$.
%%, with $j=0,\,1$.
We choose $\bm{n}_0,\,\bm{n}_1\in\C^2$
such that
$\abs{\bm{n}_0}=\abs{\bm{n}_1}=1$,
$\bm{n}_0\sp\ast\bm{n}_1=0$;
similarly,
we choose $\bm{m}_0,\,\bm{m}_1\in\C^2$
such that
$\abs{\bm{m}_0}=\abs{\bm{m}_1}=1$,
$\bm{m}_0\sp\ast\bm{m}_1=0$;
then
\[
\phi_i^*\beta\phi_j=0,
\qquad
\chi_i^*\beta\chi_j=0,
\qquad
i,\,j=0,\,1,
\ \quad
%\mbox{ as long as } 
i\ne j.
\]
Taking into account the relations
\[
\phi_j^*\beta\phi_j=|v_j(r)|^2-|u_j(r)|^2
\quad
\mbox{and}
\quad
\chi_j^*\beta\chi_j=|u_j(r)|^2-|v_j(r)|^2,
\qquad
j=0,\,1,
\]
we derive:
\begin{eqnarray*}
F(r)
&:=&
\psi(x,t)^*\beta\psi(x,t)
\\
&=&
(\abs{a_0}^2-\abs{b_0}^2)\big(v_0(r)^2-u_0(r)^2\big)
+
(\abs{a_1}^2-\abs{b_1}^2)\big(v_1(r)^2-u_1(r)^2\big).
\end{eqnarray*}
In view of Assumption~\ref{ass-mono},
the function $F(r)$
is positive, differentiable, and strictly monotonically decreasing
to zero as $r\to\infty$
as long as
$a_0,\,a_1,\,b_0,\,b_1\in\C$ are
chosen so that $\abs{a_0}^2>\abs{b_0}^2$
and so that $\abs{\abs{a_1}^2-\abs{b_1}^2}$ is sufficiently small.
Therefore,
there is a differentiable monotonically increasing function $f$
with $f(0)=0$
such that
$f(F(r))=V(r)$.
Then the four-frequency solitary wave
$\psi(x,t)$ from \eqref{def-psi-a-b}
satisfies the nonlinear Dirac equation \eqref{nld}.

\begin{remark}
We do not know whether
the potential $W$ in \eqref{one-and-a-half}
could be chosen so that
the function $f$
obtained in the above construction
is polynomial or algebraic.
\end{remark}

\bibliographystyle{sima-doi-sc}
\bibliography{bibcomech}
\end{document}